\documentclass[a4paper,12pt,final]{amsart}
\usepackage{times, a4wide,mathrsfs,amssymb,dsfont, enumerate, xypic,graphicx}

\newcommand{\C}{\mathbb{C}}
\newcommand{\Z}{\mathbb{Z}}

\newcommand{\QQ}{\mathbb{Q}}

\newcommand{\PP}{\mathbb{P}}

\newcommand{\A}{\mathbb{A}}

\newcommand{\DD}{\mathcal D}
\newcommand{\WW}{\mathcal W}

\newcommand{\XX}{\mathcal X}

\newcommand{\CC}{\mathcal C}

\newcommand{\MM}{\mathcal M}

\newcommand{\hh}{\mathfrak h}
\newcommand{\ttt}{\mathfrak t}

\newcommand{\ima}{\hbox{Im}}

\newcommand{\Gr}{\operatorname{Gr}}
\newcommand{\QQQ}{\mathcal Q}
\newcommand{\Hom}{\operatorname{Hom}}

\newcommand{\DM}{\operatorname{DM}}

\newcommand{\one}{\mathds{1}}

    \newcommand\dashdownarrowi{\mathchoice%
    {\rotatebox[origin=c]{-90}{$\displaystyle\dashrightarrow$}}%
    {\rotatebox[origin=c]{-90}{$\displaystyle\dashrightarrow$}}%
    {\rotatebox[origin=c]{-90}{$\scriptstyle\dashrightarrow$}}%
    {\rotatebox[origin=c]{-90}{$\scriptscriptstyle\dashrightarrow$}}%
} 

\newcommand{\dashdownarrow}{\mathrel{\dashdownarrowi}}

\newtheorem{theorem}{Theorem}[section]
\newtheorem{claim}[theorem]{Claim}
\newtheorem{lemma}[theorem]{Lemma}

\newtheorem{proposition}[theorem]{Proposition}

\newtheorem{convention}{Conventions}

\theoremstyle{definition}
\newtheorem{remark}[theorem]{Remark}
\newtheorem{definition}[theorem]{Definition}

\newtheorem{thm2}{Theorem}

\newtheorem{nonumberingt}{Acknowledgements}

\begin{document}
\author[Robert Laterveer]
{Robert Laterveer}

\address{Institut de Recherche Math\'ematique Avanc\'ee,
CNRS -- Universit\'e 
de Strasbourg,\
7 Rue Ren\'e Des\-car\-tes, 67084 Strasbourg CEDEX,
FRANCE.}
\email{robert.laterveer@math.unistra.fr}

\title{On the Chow ring of Fano varieties of type $S6$}

\begin{abstract} Fatighenti and Mongardi have defined Fano varieties of type S6 as zero loci of a certain vector bundle on the Grassmannian $\Gr(2,10)$.
These varieties have 3 Hodge structures of K3 type in their cohomology. We show that the Chow ring of these varieties also displays ``K3 type'' behaviour.
\end{abstract}



\maketitle

\section*{Introduction}

 For a smooth projective variety $X$ over $\C$, we write $A^i(X)=CH^i(X)_{\QQ}$ for the Chow group of codimension $i$ algebraic cycles modulo rational equivalence with $\QQ$-coefficients, and $A^i_{hom}(X)$ for the subgroup of homologically trivial cycles. Intersection product defines a ring structure on $A^\ast(X)=\oplus_i A^i(X)$. In the case of K3 surfaces, this ring structure has a curious property:

\begin{thm2}[Beauville--Voisin \cite{BV}]\label{K3} Let $S$ be a projective K3 surface. 
 Let $R^\ast(S)\subset A^\ast(S)$ be the $\QQ$-subalgebra generated by $A^1(S)$ and the Chow-theoretic Chern class $c_2(S)$. The cycle class map induces an injection
   \[ R^\ast(S)\ \hookrightarrow\ H^\ast(S,\QQ)\ .\]
  \end{thm2}
   
 This note is about the Chow ring of Fano varieties {\em of type S6\/} in the terminology of Fatighenti--Mongardi \cite{FM}. By definition, a Fano variety of type S6 is the smooth dimensionally transverse zero locus of a global section of the bundle $\QQQ^\ast(1)$ on the Grassmannian $\Gr(2,10)$ (here $\QQQ$ is the universal quotient bundle). Varieties of type S6 are 8-dimensional, and their Hodge diamond looks like
  \[ \begin{array}[c]{ccccccccc}
	&&&&1&&&&\\
	&&&  &1&  &&&\\
	&&&&2&&&&\\
	&&&1&22&1&&&\\	
	0& \dots&\dots&1&23&1&\dots&\dots&0\\
	&&&1&22&1&&&\\
	&&&  &2&  &&&\\
	&&&&1&&&&\\
	&&&&1&&&&\\
\end{array}\]
(where all empty entries are $0$).	   
 The remarkable thing is that these varieties have 3 Hodge structures of K3 type in their cohomology. As shown in \cite[Proposition 3.29]{FM} (cf. also \cite[Section 4]{BFM}), general varieties of type S6 are related to the hyperk\"ahler fourfolds of Debarre--Voisin \cite{DV}, which somehow explains these Hodge structures.

The goal of this note is to see how the peculiar shape of this Hodge diamond translates into peculiar properties of the Chow ring.
A first result is as follows:
 
 \begin{thm2}\label{main} Let $X$ be a Fano variety of type S6. Then
 \[ A^i_{hom}(X)=0\ \ \forall i\not\in \{ 4,5,6\}\ .\]
 Moreover, for $X$ general intersecting with an ample divisor $h$ induces isomorphisms
  \[ \begin{split} \cdot h\colon\ \ &A^4_{hom}(X)\ \xrightarrow{\cong}\ A^5_{hom}(X)\ ,\\
             \cdot h\colon\ \ &A^5_{hom}(X)\ \xrightarrow{\cong}\ A^6_{hom}(X)\ .\\
             \end{split}\]
       \end{thm2}
       
 This is in accordance with the Bloch--Beilinson conjectures \cite{J2}. (Indeed, the fact that $A^8_{hom}(X)=A^7_{hom}(X)=0$ corresponds to the fact that 
 $h^{p,q}(X)=0$ for $p\not=q, p+q\le 4$. The fact that $A^4_{hom}(X)\cong A^5_{hom}(X)\cong A^6_{hom}(X)$ corresponds to the fact that cupping with an ample divisor induces isomorphisms in transcendental cohomology $H^6_{tr}(X)\cong H^8_{tr}(X)\cong H^{10}_{tr}(X)$).  
 
 A second result concerns the ring structure of the Chow ring:
 
 \begin{thm2}\label{main2} Let $X$ be a Fano variety of type S6. Let $R^\ast(X)\subset A^\ast(X)$ denote the $\QQ$-subalgebra
   \[ R^\ast(X):=\langle A^1(X), A^2(X), c_j(X), \ima\bigl( A^j(\Gr(2,10))\to A^j(X)\bigr)\rangle\ \ \ \subset A^\ast(X)\ .\]
   The cycle class map induces an injection
   \[ R^\ast(X)\ \hookrightarrow\ H^\ast(X,\QQ)\ .\]
   \end{thm2}
    
 For $X$ general, we also prove that $A^2(X)\cdot A^3(X)$ injects into cohomology (Proposition \ref{23}).
 This is reminiscent of the behaviour of the Chow ring of a K3 surface (Theorem \ref{K3}). Theorem \ref{main2} suggests that $X$ might perhaps have a multiplicative Chow--K\"unneth decomposition in the sense of \cite{SV}, which would be a manifestation of the fact that ``$X$ is close to K3 surfaces''. Establishing this seems difficult however (cf. Remark \ref{difficult} below).      
    
\vskip0.6cm

\begin{convention} In this note, the word {\sl variety\/} will refer to a reduced irreducible scheme of finite type over the field of complex numbers $\C$. 
{\bf All Chow groups will be with $\QQ$-coefficients, unless indicated otherwise:} For a variety $X$, we will write $A_j(X):=CH_j(X)_{\QQ}$ for the Chow group of dimension $j$ cycles on $X$ with rational coefficients.
For $X$ smooth of dimension $n$, the notations $A_j(X)$ and $A^{n-j}(X)$ will be used interchangeably. 
The notation
$A^j_{hom}(X)$ will be used to indicate the subgroups of 
homologically trivial  cycles.

We will write  $\MM_{{\rm rat}}$ for the contravariant category of Chow motives (i.e., pure motives as in \cite{Sc}, \cite{MNP}).
\end{convention}

\section{Preliminaries}

\subsection{Transcendental part of the motive} We recall a classical result concerning the motive of a surface:

\begin{theorem}[Kahn--Murre--Pedrini \cite{KMP}]\label{t2} Let $S$ be a surface. There exists a decomposition
  \[ \hh^2(S)= \ttt(S)\oplus \hh^2_{alg}(S)\ \ \ \hbox{in}\ \MM_{\rm rat}\ ,\]
  such that
  \[  H^\ast(\ttt(S),\QQ)= H^2_{tr}(S)\ ,\ \ H^\ast(\hh^2_{alg}(S),\QQ)=NS(S)_{\QQ}\ \]
  (here $H^2_{tr}(S)$ is defined as the orthogonal complement of the N\'eron--Severi group $NS(S)_{\QQ}$ in $H^2(S,\QQ)$),
  and
   \[ A^\ast(\ttt(S))_{}=A^2_{AJ}(S)_{}\ .\]
   (The motive $\ttt(S)$ is called the {\em transcendental part of the motive\/}.)
   \end{theorem} 
   
 The following provides a higher-dimensional version:  

\begin{proposition} Let $X$ be a smooth projective variety of dimension $n$, and assume $X$ is a complete intersection in a variety with trivial Chow groups. There exists a decomposition
  \[ \hh(X)= \ttt(X)\oplus  \bigoplus_{j} \mathds{1}(j)\ \ \ \hbox{in}\ \MM_{\rm rat}\ , \]
such that
    \[  H^\ast(\ttt(X),\QQ)= H^\ast_{tr}(X)\ \ \]
  (here the transcendental part $H^\ast_{tr}(X)$ is defined as the orthogonal complement of the algebraic part $N^{ \ast}(X):=\ima\bigl( A^{ \ast}(X)\to H^\ast(X,\QQ)\bigr)$),
  and
   \[ A^\ast(\ttt(X))_{}=A^\ast_{hom}(X)_{}   \ .\]
   (The motive $\ttt(X)$ will be called the {\em transcendental part of the motive\/}.)
\end{proposition}

\begin{proof} This is a standard construction (cf. \cite[Section 2]{BP}, where this decomposition is constructed for cubic fourfolds). One can apply \cite[Theorem 1]{Vi}, where projectors $\pi^j_{alg}$ on the algebraic part of cohomology are constructed for any smooth projective variety satisfying the standard conjectures. The motive $\ttt(X)$ is then defined by the projector $\Delta_X-\sum_j \pi^j_{alg}$.
\end{proof}

\subsection{Voevodsky motives}

 \begin{definition} Let $\DM$ be the triangulated category
    \[ \DM:=\DM_{\rm Nis}^{\rm eff,-}(\C,\Z) \]
    as defined in \cite[Definition 14.1]{MVW}.
    
    For any (not necessarily smooth) variety $X$, let $M^c(X):= z_{equi}(X,0)\in\DM$ denote the {\em motive with compact support\/} as in 
    \cite[Definition 16.13]{MVW}. Here $z_{equi}(X,0)$ denotes the sheaf of equidimensional cycles of relative dimension $0$ \cite[Definition 16.1]{MVW}, considered as object of $\DM$. 
    
  Given $M\in\DM$ and $n\in\Z$, there are objects $M[n]\in \DM$ (this corresponds to shifting the degrees of the complex defining $M$), and $M(n)\in\DM$ (this corresponds to the ``Tate twist'' in $\MM_{\rm rat}$.
  \end{definition}
  
  \begin{remark}\label{remark} There is a contravariant fully faithful functor
    \[ \MM_{\rm rat}\ \to\ \DM\ ,\]
    sending the motive $h(X)=(X,\Delta_X,0)$ to the motive $M^c(X)$ and sending $\one(n)$ to $\one(n)[2n]$ \cite[Chapter 20]{MVW}.

  If $X$ is smooth projective, there is an isomorphism $M(X)\cong M^c(X)$ \cite[Example 16.2]{MVW}, where the motive $M(X)$ is as in \cite[Definition 14.1]{MVW}.
  
  For any closed immersion $Y\subset X$ with complement $U:=X\setminus Y$, there is a ``Gysin'' distinguished triangle in $\DM$
    \[ M^c(Y)\ \to\ M^c(X)\ \to\ M^c(U)\ \xrightarrow{[1]}\  \]
    \cite[Theorem 16.15]{MVW}.
    
    The functor $M^c(-)$ is contravariant with respect to \'etale morphisms, and covariant for proper morphisms. 
    Given $X$ smooth and $h\in A^1(X)$ ample, one can define functorial maps $\cdot h^j\colon M^c(X)\to M^c(X)(j)[2j]$, that correspond to intersecting with $h^j$.
   \end{remark}

\begin{proposition}\label{projbun} Let $X$ be a smooth quasi-projective variety, and let $p\colon P\to X$ be a $\PP^r$-bundle. Let $h\in A^1(P)$ be ample. There is an isomorphism in $\DM$
  \[  \alpha:=  \sum_{j=0}^r  p_\ast \cdot  h^j\colon\ \  M^c(P)\ \xrightarrow{\cong}\ \bigoplus_{j=0}^r M^c(X)(j)[2j]\ .\]
\end{proposition}

\begin{proof} This is dual to the projective bundle formula for $M(X)$ given in \cite[Theorem 15.12]{MVW}.
Using the Gysin distinguished triangle \cite[Exercice 16.18]{MVW}, one reduces to the case $P=X\times\PP^r$. Using the isomorphism $M^c(X\times\PP^r) \cong M^c(X)\otimes M^c(\PP^r)$ \cite[Corollary 16.16]{MVW}, one reduces to the case where $X$ is a point, which follows from $M^c(\A^i)=\one(i)[2i]$.
\end{proof}

\begin{lemma}\label{zero}
 Let $X,Y$ be smooth projective varieties, let $j\in\Z$ and $k\in\Z_{\ge 0}$. Then
  \[  \Hom_{_{\DM}}\bigl(M(X)(k)[2k], M(Y){ (j)[2j+1]}\bigr)=0\ .\]
  \end{lemma}
  
\begin{proof}
 For any smooth projective varieties $X, Y$ with $d:=\dim Y$, and any $k\ge 0$ one has
     \[    \begin{split} \Hom_{_{\DM}}\bigl(M(X)(k)[2k], M(Y){ (j)[2j+1]}\bigr)&= \Hom_{_{\DM}}\bigl(M(X\times Y)(k-d)[2k-2d], \one(j)[2j+1]\bigr)  \\
                          &=\Hom_{_{\DM}}\bigl( M(X\times Y), \one(j-k+d)[2(j-k+d)+1]\bigr)\\
                                &= H^{2(j-k+d)+1\, ,\, j-k+d}(X,\Z)\\
                                &= A^{j-k+d}(X,-1)=0\ .\\
                                \end{split}     \]
                                Here, the first equality follows from the duality theorem \cite[Theorem 16.24]{MVW}, the second equality is cancellation \cite[Theorem 16.25]{MVW}, the third equality is by definition of motivic cohomology $H^{\ast,\ast}(-,\Z)$ for the smooth variety $X$ \cite[14.5]{MVW}, and the last equality follows from the relation between motivic cohomology and higher Chow groups $A^\ast(-,\ast)$ \cite[Theorem 19.1]{MVW}.

\end{proof}

\section{First result}

This section contains the proof of Theorem \ref{main} stated in the introduction.


\begin{proof}(of Theorem \ref{main}) The argument is based on a nice geometric relation between $X$ and the $20$-dimensional Debarre--Voisin hypersurface $X_{DV}\subset\Gr(3,10)$ \cite{DV}. This geometric relation is described in \cite[Section 3.10]{FM} (cf. also \cite[Section 4]{BFM}, where $X$ is called $T=T(2,10)$). A Debarre--Voisin hypersurface is by definition a smooth hyperplane section $X_{DV}\subset\Gr(3,10)$ (with respect to the Pl\"ucker embedding). Starting from an $X_{DV}$, we construct the diagram (\ref{diag}) below.
Here $\operatorname{Fl}(2,3,10)$ denotes the flag variety, and the morphism $p_{Fl}$ is a $\PP^2$-bundle (the fibres correspond to a choice of $2$-dimensional subvector space in a fixed $3$-dimensional vector space). The morphism $p$ induced by restricting $p_{Fl}$ to $Z:=(p_{Fl})^{-1}(X_{DV})$ is again a $\PP^2$-bundle.
The projection $\phi_{Fl}$ from $\operatorname{Fl}(2,3,10)$ to $\Gr(2,10)$ is a $\PP^6$-bundle. The restriction $\phi:=\phi_{Fl}\vert_Z$ is therefore generically a $\PP^6$-bundle. The locus over which $\phi$ has $7$-dimensional fibres is the zero locus $X$ of a section of the dual of $\QQQ(-1)$. Hence, if this locus $X$ is smooth and dimensionally transverse, it is a variety of type S6.

\begin{equation}\label{diag} 
\begin{array}[c]{ccccccccc}
     &&Z_X&\xrightarrow{\iota}& Z &\hookrightarrow{}& \operatorname{Fl}(2,3,10)&&\\
     &&&&&&&&\\
     &\swarrow{\scriptstyle \phi_X}&& \swarrow{\scriptstyle \phi} &&\ \ \ \searrow{\scriptstyle p}&&\ \ \ \searrow {\scriptstyle p_{Fl}} &\\
     &&&&&&&&\\
     X&\hookrightarrow&\Gr(2,10)&&&&X_{DV}&\hookrightarrow&\Gr(3,10)\\
     \end{array}
     \end{equation}

As explained in \cite{FM}, there is an isomorphism
  \[ H^0(\Gr(2,10),\QQQ(1))\cong \wedge^3 V_{10}^\vee\ ,\]
 and so the space parametrizing Fano varieties of type S6 is of the same dimension as the space parametrizing Debarre--Voisin hypersurfaces. It follows 
 that a general Fano variety $X$ of type S6 can be obtained from a diagram (\ref{diag}). 

 We proceed to relate $X$ and $X_{DV}$ on the level of motives:

\begin{theorem}\label{main0} Let $X$ be a Fano variety of type S6, and assume $X$ is related to a Debarre--Voisin hypersurface $X_{DV}$ as in diagram (\ref{diag}). Then there is an isomorphism of Chow motives
  \[ h(X)\cong \bigoplus_{j= -7}^{-5} \ttt(X_{DV})(j)\oplus \bigoplus \mathds{1}(\ast)\ \ \ \hbox{in}\ \MM_{\rm rat}\ .\]
 \end{theorem}

\begin{proof} 
Let us write $G:=\Gr(2,10)$ and $U:=G\setminus X$ and $Z_U:=Z\setminus Z_X$. Also, let us write $\phi_U\colon Z_U\to U$ for the restriction of $\phi$ to $Z_U$.
The inclusion $Z_X\hookrightarrow Z$ gives rise to a distinguished triangle
 \[  M^c(Z_X)\to M^c(Z) \to M^c(Z_U) \xrightarrow{[1]} \]
  in $\DM$. Since $Z_U\to U$ is a $\PP^6$-bundle, there is an isomorphism $M^c(Z_U)\cong\oplus_{j=0}^6 M^c(U)(j)[2j]$ (Proposition \ref{projbun}). It follows there is also a distinguished triangle
  \[   M^c(Z_X)\ \to\ M^c(Z)\  \to\ \bigoplus_{j=0}^6 M^c(U)(j)[2j] \ \xrightarrow{[1]} \ ,\]
 and after rotating one obtains a distinguished triangle in $\DM$  
  \begin{equation}\label{first}   \bigoplus_{j=0}^6 M^c(U)(j)[2j-1]\to M^c(Z_X)\to M^c(Z)  \xrightarrow{[1]} \ .\end{equation}

  We claim that the triangle (\ref{first}) fits into a commutative diagram:
  
  \begin{claim}\label{claim} There is a commutative diagram in $\DM$
      \begin{equation}\label{dicl}  \begin{array}[c]{cccccc}
{\displaystyle\bigoplus_{j=0}^6}  M^c(U)(j)[2j-1] & \to &     M^c(Z_X)    & \to &    M^c(Z)    &  \xrightarrow{[1]} \\
               &&&&&\\
                                              & \searrow &  \downarrow      &     &&\\
                                              &&&&&\\
                                              &&  {\displaystyle\bigoplus_{j=0}^6}  M^c(X)(j)[2j] & &\dashdownarrow& \\
                                              &&&&&\\
                                              &&&&&\\
                                              &&\downarrow&\searrow\ \ \ &&\\
                                              &&&&&\\
                                              &&&&&\\
                                                          &&   M^c(X)(7)[15] &\dashleftarrow&  {\displaystyle\bigoplus_{j=0}^6 } M^c(G)(j)[2j]      &\\
                                              &&&&&\\
                                              &&\ \ \downarrow{\scriptstyle [1]}&& &\! \! \! \! \! \! \! \! \! \! \searrow^{\scriptstyle [1]}\ \ \ \ \ \ \  \ \ \ \ \ \ \ \ \ \\
                                              \end{array}\end{equation}   
    where the three lines with solid arrows are distinguished triangles.
    \end{claim}
    
 Granting this claim, one readily proves Theorem \ref{main0}: applying the octahedral axiom to the diagram (\ref{dicl}), one obtains a distinguished triangle
 following the dotted arrows:
   \[ M^c(Z)\to  \bigoplus_{j=0}^6  M^c(G)(j)[2j] \to   M^c(X)(7)[15]   \xrightarrow{[1]}\ .\]
   Applying Lemma \ref{zero}, the second arrow in this triangle must be zero, and so there is a direct sum decomposition
   \[   M^c(Z)\cong   M^c(X)(7)[14]    \oplus\bigoplus_{j=0}^6  M^c(G)(j)[2j] \ \ \ \ \hbox{in}\ \DM\ .\]
  Since all varieties in this isomorphism are smooth projective, and $\MM_{\rm rat}\to \DM$ is a full embedding, this means there is also a 
   direct sum decomposition
   \begin{equation}\label{iso7}  h(Z)\cong   h(X)(7)  \oplus\bigoplus_{j=0}^6  h(G)(j)\ \ \ \ \hbox{in}\ \MM_{\rm rat}\ .\end{equation}
   (Alternatively, the isomorphism (\ref{iso7}) can also be obtained directly by applying \cite[Corollary 3.2]{Ji}, which does not use Voevodsky motives.)
   
   The variety $Z$ is a $\PP^2$-bundle over $X_{DV}$, and so one gets an isomorphism of motives
    \[  \bigoplus_{i=0}^2 h(X_{DV})(i)\cong   h(X)(7)  \oplus\bigoplus_{j=0}^6  h(G)(j)\ \ \hbox{in}\ \MM_{\rm rat}\ .\]   
    Taking the transcendental parts on both sides (and remembering that the motive of the Grassmannian $G$ is a sum of twisted Lefschetz motives), one gets
    an isomorphism
    \[  \ttt(X)\cong  \bigoplus_{i=0}^2 \ttt(X_{DV})(-7+i) \ \ \hbox{in}\ \MM_{\rm rat}\ ,\]  
    which implies Theorem \ref{main0}.
    
    It remains to prove Claim \ref{claim}. The horizontal line of (\ref{dicl}) is the distinguished triangle (\ref{first}). The diagonal line is obtained from the ``Gysin'' distinguished triangles
     \[  M^c(X)(j)[2j]\ \to\  M^c(G)(j)[2j]\ \to\ M^c(U)(j)[2j]  \ \xrightarrow{[1]} \]
     by rotation and summing. The vertical line is a distinguished triangle because $Z_X\to X$ is a $\PP^7$-bundle (Proposition \ref{projbun}). To check commutativity, one observes there is a commutative diagram
   \begin{equation}\label{comdia}  \begin{array}[c]{cccccc}
      M^c(Z_X) &\xrightarrow{}& M^c(Z) &\to& M^c(Z_U)  &\xrightarrow{[1]}    \\
      &&&&&\\
     \ \ \ \  \downarrow{\scriptstyle\alpha_X}  &&  \ \ \ \  \downarrow{\scriptstyle\alpha}   &&    \ \ \ \ \ \ \cong\ \  \downarrow{\scriptstyle\alpha_U}&\\
        &&&&&\\
        \bigoplus_{j=0}^{6}M^c(X)(j)[2j] &\xrightarrow{}& \bigoplus_{j=0}^{6} M^c(G)(j)[2j]&\to& \bigoplus_{j=0}^6 M(U)(j)[2j]&\xrightarrow{[1]}\\
        \end{array}      \end{equation}
        Here the map $\alpha$ is defined as a sum $\sum_{j=0}^6 \phi_\ast \cdot h^j$ where $h$ is an ample class, and the maps $\alpha_X, \alpha_U$ are defined similarly by restricting $h$ to $X$ resp. to $U$. The map $\alpha_U$ is an isomorphism because $Z_U\to U$ is a $\PP^6$-bundle.
        (The left square commutes by functoriality of proper push-forward. The commutativity of the right square follows from commutativity on the level of $z_{equi}(-,0)$ , which can be checked directly.) After rotation, diagram (\ref{comdia}) gives a commutative diagram
        \begin{equation}\label{comdia2}  \begin{array}[c]{cccccc}
      M^c(Z_U)[-1] &\xrightarrow{}& M^c(Z_X) &\to& M^c(Z)  &\xrightarrow{[1]}    \\
      &&&&&\\
     \ \ \ \ \ \  \cong\ \ \downarrow{\scriptstyle\alpha_U}  &&  \ \ \ \  \downarrow{\scriptstyle\alpha_X}   &&    \ \ \ \  \downarrow{\scriptstyle\alpha}&\\
        &&&&&\\
        \bigoplus_{j=0}^{6}M^c(U)(j)[2j-1] &\xrightarrow{}& \bigoplus_{j=0}^{6} M^c(X)(j)[2j]&\to& \bigoplus_{j=0}^6 M^c(G)(j)[2j]&\xrightarrow{[1]}\\
        \end{array}      \end{equation}
       In the diagram (\ref{dicl}), the horizontal line is the top horizontal line of (\ref{comdia2}) combined with the isomorphism $\alpha_U$, while the diagonal line
       in (\ref{dicl}) is the bottom horizontal line of (\ref{comdia2}). This proves the commutativity of Claim \ref{claim}.

         \end{proof}

We now pursue the proof of Theorem \ref{main}. Let $\XX\to B$ denote the universal family of Fano varieties of type S6, where $B$ is a Zariski open in
$\PP  H^0(\Gr(2,10),\QQQ(1))$.
As we have seen, a general Fano variety of type S6 fits into a diagram (\ref{diag}), and so there is a Zariski open $B_0\subset B$ to which Theorem \ref{main0} applies. Taking Chow groups of the motives in Theorem \ref{main0}, we get an isomorphism
  \begin{equation}\label{isochow} A^i_{hom}(X_b)\cong A^{i+7}_{hom}(X_{DV})\oplus A^{i+6}_{hom}(X_{DV})\oplus A^{i+5}_{hom}(X_{DV})\ \ \ \forall b\in B_0\ .\end{equation}
  
As an illustration of her celebrated method of spread, Voisin has proven the following:

\begin{theorem}[Voisin \cite{V1}]\label{cv} Let $X_{DV}$ be a Debarre--Voisin hypersurface. Then
  \[ A^i_{hom}(X_{DV})=0\ \ \ \forall i\not=11\ .\]
  \end{theorem}
  
  (More precisely, Voisin \cite[Theorem 2.4]{V1} proves $A^i_{hom}(X_{DV})=0$ for $i>11$, but this readily
   implies that $A^i_{hom}(X_{DV})=0$ for $i<11$ as well, as explained in \cite[Proof of Theorem 2.1]{35}.)
  Plugging in Theorem \ref{cv} into isomorphism (\ref{isochow}), we find that
  \begin{equation}\label{vani}  A^i_{hom}(X_b)=0\ \ \forall i\not\in\{ 4,5,6\}\ \ \ \forall b\in B_0\ .\end{equation}
  We now proceed to extend this vanishing from $B_0$ to all of $B$. To do this, we observe that the vanishing (\ref{vani}) implies, via the Bloch--Srinivas ``decomposition of the diagonal'' method \cite{BS}, that for any $b\in B_0$ there is a decomposition of the diagonal
  \begin{equation}\label{decomp} \Delta_{X_b}= x\times X_b + C_b\times D_b +\Gamma_b\ \ \ \hbox{in}\ A^{8}(X_b\times X_b)\ ,\end{equation}
  where $x\in X_b$, $C_b$ and $D_b$ are a curve resp. a divisor, and $\Gamma_b$ is a cycle supported on $X_b\times W_b$ where $W_b\subset X_b$ is a codimension $2$ closed subvariety.
  Using the Hilbert schemes argument of \cite[Proposition 3.7]{V0}, these data can be spread out over the base $B$. That is, we can find $x\in A^8(\XX)$, $\CC\in A^7(\XX)$, $\DD\in A^1(\XX)$, a codimension $2$ subvariety $\WW\subset \XX$ and a cycle $\Gamma$ supported on $\XX\times_B \WW$ such that
  restricting to a fibre $b\in B_0$ we get
  \[ \Delta_{X_b}= x\vert_b\times X_b + \CC\vert_b\times \DD\vert_b +\Gamma\vert_b\ \ \ \hbox{in}\ A^{8}(X_b\times X_b)\ .\] 
  Applying \cite[Lemma 3.2]{Vo}, we find that this decomposition is actually true for any $b$ in the larger base $B$.
   What's more, given any $b_0\in B$, this construction can be done in such a way that $\CC,\DD,\WW$ are in general position with respect to the fibre $X_{b_0}$. That is, we have obtained a decomposition (\ref{decomp}) for any $b_0\in B$. Letting this decomposition act on Chow groups, it follows that
   the vanishing (\ref{vani}) is true for all $b\in B$. This proves the first part of Theorem \ref{main}.
   
   We now prove the ``moreover'' part of Theorem \ref{main}. 
   So let us consider a Fano variety $X=X_b$ for $b\in B_0$, which means that we can assume $X$ fits into a diagram (\ref{diag}) with some $22$-dimensional variety $Z$ and a Debarre--Voisin hypersurface $X_{DV}$. The isomorphism (\ref{iso7}) implies that there are isomorphisms
   \begin{equation}\label{456}  \iota_\ast (\phi_X)^\ast\colon\ \ A^i_{hom}(X)\ \xrightarrow{\cong}\ A^{i+7}_{hom}(Z)\ \ \ \ (i\in\{4,5,6\})\ .\end{equation}
 The Picard number of $X$ being $1$, any ample divisor $h\in A^1(X)$ comes from a divisor $h_G\in A^1(\Gr(2,10))$. Let $h_Z:=\phi^\ast(h_G)\in A^1(Z)$. 
 In view of the isomorphisms (\ref{456}), to prove the ``moreover'' statement it suffices to prove there are isomorphisms
 \begin{equation}\label{hZ} \begin{split} &\cdot h_Z\colon\ \  A^{11}_{hom}(Z)\ \xrightarrow{\cong}\ A^{12}_{hom}(Z)\ ,\\
               &\cdot h_Z\colon\ \  A^{12}_{hom}(Z)\ \xrightarrow{\cong}\ A^{13}_{hom}(Z)\ .\\  
               \end{split}\end{equation}
Since $p\colon Z\to X_{DV}$ is a $\PP^2$-bundle, the divisor $h_Z\in A^1(Z)$ can be written
  \[ h_Z= p^\ast(c)+\lambda\, \xi\ \ \ \hbox{in}\ A^1(Z)      \ ,\]
  where $\lambda\in\QQ$ and $\xi$ is a relatively ample class for the fibration $p$. We claim that $\lambda$ must be non-zero. (Indeed, suppose $\lambda$ were zero. Then the intersection of $(h_Z)^2$ with a fibre $F$ of $p$ would be zero. But the image $\phi(F)\subset \Gr(2,10)$ is a $2$-dimensional subvariety and so $(h_G)^2\cdot \phi(F)\not=0$; contradiction.)       
  
 The fact that $p$ is a $\PP^2$-bundle, plus the fact that $A^j_{hom}(X_{DV})=0$ for $j\not=11$, gives us isomorphisms
 \[ \begin{split}  &A^{11}_{hom}(Z)\cong p^\ast A^{11}_{hom}(X_{DV})\ ,\\ &A^{12}_{hom}(Z)\cong p^\ast A^{11}_{hom}(X_{DV})\cdot\xi\ ,\\   &A^{13}_{hom}(Z)\cong p^\ast A^{11}_{hom}(X_{DV})\cdot \xi^2\    .\\
 \end{split}\]
 Thus, we see that $\cdot p^\ast(c)$ acts as zero on $A^{11}_{hom}(Z)$ and on $A^{12}_{hom}(Z)$  (indeed, this map factors over $A^{12}_{hom}(X_{DV})=0$).
 It follows that intersecting with $h_Z$ is the same as intersecting with $\lambda\, \xi$ on $A^{j}_{hom}(Z)$, and this induces the desired isomorphisms (\ref{hZ}). We have now proven the ``moreover'' statement for $X_b$ with $b\in B_0$.
                 \end{proof}

\section{Second result}

In this section we prove Theorem \ref{main2} stated in the introduction.
In order to prove Theorem \ref{main2}, we first establish a ``Franchetta property'' type of statement (for more on the generalized Franchetta conjecture, cf. \cite{OG}, \cite{PSY}, \cite{FLV}):

\begin{theorem}\label{gfc} Let $\XX\to B$ denote the universal family of Fano varieties of type S6 (as above).
Let $\Psi\in A^{j}(\XX)$ be such that
  \[ \Psi\vert_{X_b}=0\ \ \hbox{in}\ H^{2j}(X_b,\QQ)\ \ \ \forall b\in B\ .\]
  Then
   \[ \Psi\vert_{X_b}=0\ \ \hbox{in}\ A^{j}(X_b)\ \ \ \forall b\in B\ .\]
\end{theorem}

\begin{proof} Invoking the spread lemma \cite[Lemma 3.2]{Vo}, it will suffice to prove the theorem over the Zariski open $B_0\subset B$ of Theorem \ref{main0}.
The construction of diagram (\ref{diag}) being geometric in nature, this diagram also exists as a diagram of $B_0$-schemes. Writing $\XX\to B_0$ for the universal family of varieties of type S6 (as before), and $\XX_{DV}\to B_0$ for the universal Debarre--Voisin hypersurface, this means that there exists a relative correspondence $\Gamma\in A^\ast(\XX\times_{B_0} \XX_{DV})\oplus A^\ast(\XX)$ inducing the fibrewise injections (given by Theorem \ref{main0})
  \[ (\Gamma\vert_b)_\ast  \colon\ \ A^j(X_b)\ \hookrightarrow\ A^{11}((X_{DV})_b)\oplus \bigoplus \QQ\ .\]
 Thus, Theorem \ref{gfc} is implied by the Franchetta property for $\XX_{DV}$, which is \cite[Theorem 3.2]{35}.
    \end{proof}

It remains to prove Theorem \ref{main2}:

\begin{proof}(of Theorem \ref{main2}) Clearly, the Chern classes $c_j(X):=c_j(T_X)$ are universally defined: for any $b\in B$, we have
  \[ c_{j}(T_{X_b})= c_{j}(T_{\XX/B})\vert_{X_b}\ .\]
  Also, the image
  \[ \ima \bigl( A^{j}(\operatorname{Gr}(2,{10}))_{}\to A^{j}(X_b)\bigr)  \]
  consists of universally defined cycles (for a given $a\in A^{j}(\operatorname{Gr}(2,{10}))$, the relative cycle
  \[   (a\times B)\vert_\XX\ \ \in\ A^{j}(\XX) \]
  does the job).
 
 Since $A^1(X_b)$ is generated by a hyperplane section, clearly $A^1(X_b)$ is universally defined.
  Similarly, the fact that $A^2_{hom}(X_b)=0$, combined with weak Lefschetz in cohomology, implies that
  \[ A^2(X_b)=\ima\bigl(  A^{2}(\operatorname{Gr}(2,{10}))_{}\to A^{2}(X_b)\bigr)  \ ,\]
  and so $ A^2(X_b)$ also consists of universally defined cycles. 
%
%

Intersections of universally defined cycles are universally defined, since $A^\ast(\XX)\to A^\ast(X_b)$ is a ring homomorphism.
In conclusion, we have shown that $R^{\ast}(X_b)$ consists of universally defined cycles, and so Theorem \ref{main2} is a corollary of Theorem \ref{gfc}.
    \end{proof}

%

\begin{remark}\label{difficult} Theorem \ref{main2} is an indication that maybe varieties $X$ of type S6 have a {\em multiplicative Chow--K\"unneth decomposition\/}, in the sense of \cite[Chapter 8]{SV}. Unfortunately, establishing this seems difficult; one would need something like theorem \ref{gfc} for
  \[ A^{16}(\XX\times_B \XX\times_B \XX)\ .\]
\end{remark}

Presumably, one can also add $A^3(X)$ to the subring $R^\ast(X)$ of Theorem \ref{main2}. (Indeed, $A^3_{hom}(X)=0$, so {\em provided\/} $X$ has a multiplicative Chow--K\"unneth decomposition, one would have $A^3(X)=A^3_{(0)}(X)$ where $A^\ast_{(\ast)}(X)$ indicates the bigrading induced by the multiplicative Chow--K\"unneth decomposition.) While I cannot prove this, I can prove at least a weaker result:

\begin{proposition}\label{23} Let $X$ be a general Fano variety of type S6. Then
  \[ A^2(X)\cdot A^3(X)\ \subset\ A^5(X) \]
  injects into $H^{10}(X,\QQ)$ under the cycle class map.
  \end{proposition}
  
  \begin{proof} Assume $X=X_b$ with $b\in B_0$, so that $X$ is related to a Debarre--Voisin hypersurface $X_{DV}$ as in diagram (\ref{diag}). 
  We want to prove that
    \[ \Bigl( A^2(X)\cdot A^3(X)\Bigr)\cap A^5_{hom}(X)=0\ .\]
  Since $\phi_X\colon Z_X\to X$ is a $\PP^7$-bundle, it will suffice to prove that
    \[ \Bigl((\phi_X)^\ast A^2(X)\cdot (\phi_X)^\ast A^3(X)\Bigr)\cap A^5_{hom}(Z_X)=0\ .\]
   Restriction induces an isomorphism $\iota^\ast\colon A^2(Z)\cong A^2(Z_X)$. Moreover, we know (isomorphism (\ref{456})) that
   \[ \iota_\ast (\phi_X)^\ast\colon\ \ A^5_{hom}(X)\ \to\ A^{12}_{hom}(Z)\ \]
   is injective.
   Thus, it suffices to prove that
   \[ \Bigl( A^2(Z)\cdot A^{10}(Z)\Bigr)\cap A^{12}_{hom}(Z)=0\ .\]
   But this follows from the $\PP^2$-bundle structure of $p\colon Z\to X_{DV}$: indeed, any $a\in A^2(Z)$ and $b\in A^{10}(Z)$ can be written
   \[ \begin{split} & a=p^\ast(a_2) + p^\ast(a_1)\cdot\xi + p^\ast(a_0)\cdot \xi^2\ \ \ \hbox{in}\ A^2(Z)\ ,\\
                      & b= p^\ast(b_{10}) + p^\ast(b_9)\cdot\xi + p^\ast(b_8)\cdot \xi^2\ \ \ \hbox{in}\ A^{10}(Z)\ ,\\
                      \end{split}\]
                      where $\xi$ is a relatively ample class, and $a_j,b_j\in A^j(X_{DV})$. The intersection $a\cdot b$ can be written
             \[ a\cdot b= p^\ast(a_2\cdot b_{10})+p^\ast(a_1\cdot b_{10}+a_2\cdot b_9)\cdot\xi + p^\ast(a_2\cdot b_8+a_1\cdot b_9 + a_0\cdot b_{10})\cdot\xi^2\ \ \ \hbox{in}\ A^{12}(Z)\ .\]
  As the intersection $a\cdot b$ is assumed to be homologically trivial, this means that
  \[ a_2\cdot b_{10}\ ,\ \ \     a_1\cdot b_{10}+a_2\cdot b_9\ , \ \ \ a_2\cdot b_8+a_1\cdot b_9 + a_0\cdot b_{10} \]
  are homologically trivial on $X_{DV}$. But $A^{12}_{hom}(X_{DV})=A^{10}_{hom}(X_{DV})=0$, and so 
  \[ a_2\cdot b_{10} =a_2\cdot b_8+a_1\cdot b_9 + a_0\cdot b_{10} =0\ \ \ \hbox{in}\ A^\ast(X_{DV})\ .\]
  As for the remaining term, it is proven in \cite[Theorem 3.1]{35} that
  \[ A^1(X_{DV})\cdot A^{10}(X_{DV})+ A^2(X_{DV})\cdot A^9(X_{DV})\ \ \subset\ A^{11}(X_{DV}) \]
  injects into cohomology, and so also
  \[  a_1\cdot b_{10}+a_2\cdot b_9=0\ \ \ \hbox{in}\ A^{11}(X_{DV})\ .\]
  It follows that $a\cdot b=0$, and the proposition is proven.
      \end{proof}

\vskip1cm
\begin{nonumberingt} Thanks to the Lego Builders Crew of Schiltigheim for their boundless and inspiring creativity. Thanks to the referee for many constructive remarks.
\end{nonumberingt}

\end{document}